\documentclass[12pt]{amsart}

\usepackage{amsmath,amssymb,amsthm,geometry}

\usepackage{mathrsfs}

\newtheorem{theorem}{Theorem}
\newtheorem{lemma}{Lemma}[section]
\newtheorem{corollary}{Corollary}[section]

\newtheorem{remark}{Remark}[section]

\newtheorem{theoremL}{Theorem}

\newtheorem{lemmaL}{Lemma}

\newtheorem{corollaryL}{Corollary}

\theoremstyle{definition}

\DeclareMathOperator{\E}{\mathbb{E}}
\DeclareMathOperator{\Var}{\mathbb{V}}

\newcommand{\R}{\mathbb{R}}
\newcommand{\N}{\mathbb{N}}
\newcommand{\Z}{\mathbb{Z}}
\newcommand{\C}{\mathbb{C}}

\begin{document}

\title{Power Partitions and Hayman Functions}

\author{Jos\'e L. Fern\'andez}
\address{Departamento de Matem\'aticas, Universidad Aut\'onoma de Madrid, Madrid, Spain and Fundaci\'on Akusmatika.}
\email{joseluis.fernandez@uam.es}

\author{V\'{\i}ctor J. Maci\'a}
\address{Departamento de An\'alisis Matem\'atico, Universidad de La Laguna, Tenerife, Spain.}
\email{victor.macia@ull.edu.es}


\begin{abstract}
We prove, within the probabilistic framework of Khinchin families, that
the generating function~$P_k$ of partitions into $k$-th powers is strongly Gaussian in the
sense of B\'aez-Duarte, and even further that it is a Hayman function. Thus the
Hardy--Ramanujan asymptotic formula for the number~$p_k(n)$ of partitions of~$n$ into
$k$-th powers which reads
\[
p_k(n) \sim \frac{\alpha_k}{n^{(3k+1)/(2k+2)}} \exp\!\Big(\beta_k\, n^{1/(k+1)}\Big), \qquad n\to\infty,
\]
where~$\alpha_k$ and~$\beta_k$ are explicit constants depending only on~$k$, follows directly from
Hayman's asymptotic formula for strongly Gaussian power series. The proof of strong
Gaussianity of $P_k$ combines a Gaussianity criterion for Khinchin families with certain bounds of
Tenenbaum, Wu and Li on the generating function; the asymptotic formula is recovered
by computing asymptotic approximations of the mean and variance of the associated
family. Analogous results are presented for the generating function $Q_k$ of partitions into distinct $k$-th powers.
\end{abstract}

\maketitle

\vspace{-1.1em}
\begin{flushright}
	\itshape Dedicated to the memory of Christian Pommerenke.
\end{flushright}

\section*{Introduction}

For each integer $n\ge 1$, a partition of~$n$ into $k$-th powers is a non-increasing sequence
of positive integers $m_1\ge m_2\ge\ldots\ge m_l\ge 1$ so that $m_1^k+\ldots+m_l^k=n$.
\smallskip

In this paper we present a direct proof, within the probabilistic framework provided
by the theory of Khinchin families, of the asymptotic formula for the number~$p_k(n)$ of
partitions of~$n$ into perfect $k$-th powers of positive integers:
\begin{equation}\tag{HR}
p_k(n) \sim \frac{\alpha_k}{n^{(3k+1)/(2k+2)}} \exp\!\Big(\beta_k\, n^{1/(k+1)}\Big), \qquad \text{as } n\to\infty,
\end{equation}
where~$\alpha_k$ and~$\beta_k$ are specific constants which depend only on~$k$, see Theorem~\ref{thm:asymptotic of pk(n)}.
\smallskip

For general partitions, i.e., for the case $k=1$, with parameters $\alpha_1=1/(4\sqrt{3})$ and
$\beta_1=\pi\sqrt{2/3}$, this is the Hardy--Ramanujan formula from~\cite{HardyRamanujan}. The formula~(HR) for general
$k\ge 1$ also appears, without proof, in~\cite[page~111]{HardyRamanujan}, with the notation~$p^s(n)$ for partitions
into $s$-th powers.
\smallskip

Wright in~\cite{Wright} obtained an asymptotic expansion of~$p_k(n)$ as $n\to\infty$ by a rather involved argument. Recently, through an expert use of the circle method of Hardy--Ramanujan--Littlewood, Vaughan in~\cite{Vaughan} has obtained, for the case $k=2$, an asymptotic
expansion of~$p_2(n)$; this argument was later generalized by Gafni in~\cite{Gafni} to cover the
general case $k\ge 1$. More recently, Tenenbaum, Wu and Li in~\cite{TWL}, see also~\cite{TWLarxiv}, have
greatly simplified the proof of the asymptotic expansion of~$p_k(n)$ by approaching the
estimation through the saddle-point method.
\smallskip

Let us be precise about the scope of the present contribution. The asymptotic formula~(HR) and, indeed, full asymptotic expansions of~$p_k(n)$ are already known through
the works just cited. Our main goal is to show that the generating function~$P_k$ of partitions
into $k$-th powers is strongly Gaussian in the sense of B\'aez-Duarte, see Theorem~\ref{thm:strongGaussianity}, and
even further that~$P_k$ is in the Hayman class, see Theorem~\ref{thm:Hayman}; the formula~(HR) then
follows as a direct consequence, via the general Hayman asymptotic formula for strongly
Gaussian Khinchin families, Theorem~\ref{thm:hayman asympt formula}. We stress that the crucial analytic estimates underpinning the proof of strong Gaussianity are not new: they rest on the
bounds of Tenenbaum, Wu and Li recorded in Lemma~\ref{lem:bounds of char de P}. The interest of the approach,
rather, lies in the fact that it recasts the problem entirely within the probabilistic framework
of Khinchin families, where the transition from generating function to coefficient asymptotics relies on  a natural  probabilistic mechanism---verification of
strong Gaussianity and computation of the mean and variance of the associated family.
Note that this framework yields asymptotic formulas such as~(HR),
but that it does not extend to full asymptotic expansions, for which the reader is referred to
the works of Wright, Gafni, and Tenenbaum, Wu and Li cited above.
\smallskip

The theory of Khinchin families originates in the work of Hayman~\cite{Hayman}, Rosenbloom~\cite{Rosenbloom}
and B\'aez-Duarte~\cite{BaezDuarte}---whose approach we follow closely here---and has been developed
at length in~\cite{CFMK}, \cite{CFMK2} and~\cite{CFMK3}, see also~\cite{MaciaThesis} and~\cite{MaciaGaussianity}. Other instances of the use of this
approach are~\cite{CandelpergherMiniconi} and, more recently,~\cite{Ikeda}.
\smallskip

Let $P_k(z)$ be the generating function of the partitions into $k$-th powers:
\[
P_k(z) = \prod_{n=1}^{\infty} \frac{1}{1-z^{n^k}} = \sum_{n=0}^{\infty} p_k(n)\, z^n, \qquad \text{for all } z\in\mathbb{D}\,;
\]
the power series above has radius of convergence $R=1$. We denote by $(X^{[k]}_t)_{t\in[0,1)}$ the
Khinchin family associated with~$P_k$. The necessary background on Khinchin families is
reviewed in Section~\ref{sec:Khinchin}. The Gaussianity of~$P_k$ (Corollary~\ref{cor:gaussianity of Pk}) and the required asymptotic
approximations of the mean and variance of its family are established in Section~\ref{sec:partitions},
where strong Gaussianity is proved in Theorem~\ref{thm:strongGaussianity}. The asymptotic formula~(HR) is
then derived in Section~\ref{sec:asymptotic}. In Section~\ref{sec:Hayman} it is shown that~$P_k$ is in the Hayman class. In Section~\ref{sec:distinct} we discuss analogous results for the generating function~$Q_k$ of partitions into distinct $k$-th powers.

\medskip

\noindent\textbf{Some notation.}
For two functions $\alpha$ and $\beta$, we say that they are asymptotically
equivalent as $t\uparrow R$, and write $\alpha(t)\sim\beta(t)$, if
\[
\lim_{t\uparrow R} \frac{\alpha(t)}{\beta(t)} = 1\,.
\]
We use $\mathbb{P}$ to denote probability defined in the appropriate space, and denote by $\E(Y)$
and $\Var(Y)$ the expectation and variance of a random variable~$Y$.
\smallskip

We use $\mathbb{D}$ to denote the unit disk in the complex plane~$\mathbb{C}$ and $\mathbb{D}(a,r)$ to denote the
disk of center $a\in\mathbb{C}$ and radius $r>0$.

\section{Khinchin families}\label{sec:Khinchin}

We collect here the basic facts about Khinchin families that will be used throughout
this paper. Comprehensive treatments of the theory can be found in~\cite{CFMK, CFMK2, CFMK3} and also
in~\cite{MaciaThesis, MaciaGaussianity}.
\smallskip

Let $\mathcal{K}$ denote the class of nonconstant power series
\[
f(z) = \sum_{n=0}^{\infty} a_n z^n,
\]
with nonnegative coefficients, positive radius of convergence $R>0$, and $a_0>0$.
\smallskip

To each $f\in\mathcal{K}$ we associate the Khinchin family $(X_t)_{t\in[0,R)}$ of probability distributions
on the nonnegative integers $\{0,1,2,\ldots\}$ defined by
\[
\mathbb{P}(X_t = n) = \frac{a_n t^n}{f(t)}, \qquad n\ge 0,\quad t\in(0,R),
\]
and completed with $X_0\equiv 0$, for $t=0$. Since $f(t)>0$ for $t\in[0,R)$, this family of
distributions is well defined.
\smallskip

Observe that $\sigma_f(t)>0$, for $t\in(0,R)$, since each $X_t$, for $t>0$, takes at least two
values.

The mean and variance of~$X_t$ are given by
\begin{equation}\label{eq:meanvar}
\begin{aligned}
m_f(t) &= \E(X_t) = \frac{\displaystyle\sum_{n=0}^{\infty} n\, a_n\, t^n}{f(t)} = \frac{t f'(t)}{f(t)}\,,\\[6pt]
\sigma_f^2(t) &= \Var(X_t) = \E(X_t^2) - \E(X_t)^2 = \frac{\displaystyle\sum_{n=0}^{\infty} n^2 a_n t^n}{f(t)} - \left(\frac{t f'(t)}{f(t)}\right)^{\!2}\\[4pt]
&= \frac{t^2 f''(t)}{f(t)} + \frac{t f'(t)}{f(t)} - \left(\frac{t f'(t)}{f(t)}\right)^{\!2} = t\, m_f'(t)\,.
\end{aligned}
\end{equation}
The last equality follows by differentiating the expression $m_f(t)=tf'(t)/f(t)$.
\smallskip

We define the normalized random variable $\breve{X}_t$ by
\[
\breve{X}_t = \frac{X_t - m_f(t)}{\sigma_f(t)}\,, \qquad \text{for any } t\in(0,R).
\]
Note that $\breve{X}_t$ is only defined for $t\in(0,R)$, since $\sigma_f(0)=0$.

\subsection{The fulcrum of a power series}\label{subsec:fulcrum}

Every function $f\in\mathcal{K}$ is non-vanishing on the
real interval $[0,R)$. Hence for any $f\in\mathcal{K}$ we can define its so-called fulcrum~$F$ in a
simply connected domain~$\Omega_f$ containing $[0,R)$ by
\[
F(z) = \ln(f(e^z)), \qquad \text{for any } z\in\Omega_f\,,
\]
where the branch of the logarithm is chosen so that $F$ is real on $(-\infty,\ln R)$.
\smallskip

If $f$ vanishes nowhere in the unit disk (as is the case for Khinchin families associated
with partitions), then we may take $\Omega_f$ as the left half-plane $\{\Re(z)<0\}$.
\smallskip

In terms of the fulcrum we have that
\[
m_f(e^s) = F'(s), \qquad \sigma_f^2(e^s) = F''(s), \qquad \text{for any } s < \ln R.
\]
The fulcrum codifies quite efficiently many probabilistic quantities pertaining to the
Khinchin family, see~\cite{CFMK, CFMK2, CFMK3, MaciaGaussianity} for further details.

\subsection{Gaussian Khinchin families}\label{subsec:Gaussian}

We present the concept of Gaussian Khinchin family and a criterion for Gaussianity.

\subsubsection{Characteristic function}\label{subsubsec:charfn}

The characteristic function of the normalized variable $\breve{X}_t$
of the family can be written as
\[
\E(e^{\imath\theta \breve{X}_t}) = \frac{f(t e^{\imath\theta/\sigma_f(t)})}{f(t)}\, e^{-\imath\theta m_f(t)/\sigma_f(t)}, \qquad \text{for any } \theta\in\mathbb{R}.
\]
This expression connects the analytic behaviour of $f$ with the probabilistic properties
of its associated Khinchin family and plays a central role in the theory.
\smallskip

For the modulus of the characteristic function of $\breve{X}_t$ we have
\[
\big|\E(e^{\imath \theta \breve{X}_t})\big| = \frac{\big|f(t e^{\imath\theta/\sigma_f(t)})\big|}{f(t)}\,, \qquad \text{for any } \theta\in\mathbb{R}.
\]

\subsubsection{Gaussianity of a Khinchin family}\label{subsubsec:Gaussianity}

We say that $f$, or equivalently its Khinchin
family $(X_t)$, is Gaussian if $\breve{X}_t$ converges in distribution to the standard normal distribution, as $t\uparrow R$, or equivalently, by L\'evy's continuity theorem, if
\[
\lim_{t\uparrow R} \E(e^{\imath \theta \breve{X}_t}) = e^{-\theta^2/2}, \qquad \text{for all } \theta\in\mathbb{R}.
\]

\subsubsection{Criterion for Gaussianity}\label{subsubsec:criterion}

The following Gaussianity criterion in terms of the
fulcrum and its derivatives is~\cite[Theorem~4.1]{MaciaGaussianity}.

\begin{theoremL}\label{thm:gaussianty criterion}
Let $f\in\mathcal{K}$ be a power series with radius of convergence $R>0$. If its
fulcrum~$F$ satisfies
\[
\lim_{s\uparrow \ln R} \frac{F^{(j)}(s)}{F''(s)^{j/2}} = 0\,, \qquad \text{for every } j\ge 3\,,
\]
then $f$ is Gaussian.
\end{theoremL}

See~\cite{MaciaGaussianity} for a proof and a range of applications.

\subsection{Strongly Gaussian power series}\label{subsec:stronglyGaussian}

A power series $f\in\mathcal{K}$ (or its associated
Khinchin family $(X_t)$) is termed strongly Gaussian if
\[
\lim_{t\uparrow R} \sigma_f^2(t) = +\infty \qquad \text{and} \qquad \lim_{t\uparrow R} \int_{-\pi\sigma_f(t)}^{\pi\sigma_f(t)} \big|\E(e^{\imath \theta\breve{X}_t}) - e^{-\theta^2/2}\big|\, d\theta = 0.
\]
Strongly Gaussian power series satisfy a local central limit theorem that gives precise
asymptotic information about their coefficients; see~\cite{BaezDuarte, CFMK} and~\cite{Hayman}. Strong Gaussianity
(an $L^1$ condition) implies Gaussianity (a pointwise requirement). The function~$e^{z^2}$ is
Gaussian, but not strongly Gaussian, see~\cite[Remark~3.5]{CFMK}.

\subsubsection{Hayman's asymptotic formula}\label{subsubsec:Haymanformula}

In this strongly Gaussian setting, the coefficients
of $f\in\mathcal{K}$ admit a precise asymptotic description:

\begin{theoremL}[Hayman's asymptotic formula]\label{thm:hayman asympt formula}
If $f(z)=\sum_{n=0}^{\infty} a_n z^n$ in $\mathcal{K}$ is strongly
Gaussian, then
\[
a_n \sim \frac{1}{\sqrt{2\pi}}\, \frac{f(t_n)}{t_n^n\, \sigma_f(t_n)}\,, \qquad \text{as } n\to\infty,
\]
where~$t_n$ is uniquely determined by $m_f(t_n)=n$ for each $n\ge 1$.
\end{theoremL}

For strongly Gaussian families, it is always the case that $\lim_{t\uparrow R} m_f(t)=+\infty$. This
fact follows from Hayman's Central Limit Theorem; see the remarks after~\cite[Theorem~A,
Section~3.2]{CFMK}. Observe that then~$m_f$ defines a homeomorphism from $[0,R)$ onto $[0,+\infty)$,
since $t m_f'(t)=\sigma_f^2(t)>0$, for $t>0$, and thus that for each integer $n\ge 1$ there exists in
fact a unique~$t_n$ such that $m_f(t_n)=n$.

\subsubsection{B\'aez-Duarte substitution}\label{subsubsec:BDsubst}

Explicit formulas for the numbers~$t_n$ appearing in Theorem~\ref{thm:hayman asympt formula} are typically hard to obtain, since solving the equation $m_f(t)=n$ is in general
not straightforward. Fortunately, one can make do with appropriate approximations of
$m_f$ and of $\sigma_f^2$, as shown by B\'aez-Duarte~\cite{BaezDuarte}.

Assume that $f\in\mathcal{K}$ is strongly Gaussian. Let $\tilde{m}_f(t)$ be a continuous, monotonically
increasing function on $[0,R)$ with $\tilde{m}_f(t)\to+\infty$ as $t\uparrow R$, and suppose that $\tilde{m}_f(t)$
approximates~$m_f(t)$ in the sense that
\begin{equation}\label{eq:BDapprox}
\lim_{t\uparrow R} \frac{m_f(t)-\tilde{m}_f(t)}{\sigma_f(t)} = 0.
\end{equation}
For each $n\ge 1$, define~$\tau_n$ by the equation $\tilde{m}_f(\tau_n)=n$; the following version of
Theorem~\ref{thm:hayman asympt formula} holds.

\begin{theoremL}[B\'aez-Duarte substitution]\label{thm:baezduarte}
With the notations above, if $f(z)=\sum_{n=0}^{\infty} a_n z^n$
in $\mathcal{K}$ is strongly Gaussian and if~\eqref{eq:BDapprox} is satisfied, then
\[
a_n \sim \frac{1}{\sqrt{2\pi}}\, \frac{f(\tau_n)}{\tau_n^n\, \sigma_f(\tau_n)}\,, \qquad \text{as } n\to\infty.
\]
Moreover, if $\tilde{\sigma}_f(t)$ is such that $\sigma_f(t)\sim\tilde{\sigma}_f(t)$ as $t\uparrow R$, we may further write
\begin{equation}\label{eq:BDfinal}
a_n \sim \frac{1}{\sqrt{2\pi}}\, \frac{f(\tau_n)}{\tau_n^n\, \tilde{\sigma}_f(\tau_n)}\,, \qquad \text{as } n\to\infty.
\end{equation}
\end{theoremL}

See~\cite{BaezDuarte}, and also~\cite{CFMK}, for further details.

\section{Khinchin families of power partitions}\label{sec:partitions}

The infinite product
\[
P_k(z) = \prod_{n=1}^{\infty} \frac{1}{1-z^{n^k}} = \sum_{n=0}^{\infty} p_k(n)\, z^n, \qquad \text{for } |z|<1\,,
\]
is the ordinary generating function of the partitions into $k$-th powers.

Denote by $(X^{[k]}_t)_{t\in[0,1)}$ the Khinchin family associated with~$P_k$. Then, for any $t\in(0,1)$,
we have the equality in distribution
\[
X^{[k]}_t \stackrel{d}{=} \sum_{j=1}^{\infty} j^k\, G_{t^{j^k}}
\]
where $(G_u)_{u\in[0,1)}$ is the Khinchin family associated with $1/(1-z)$ and the $G_{t^{j^k}}$ in the sum
above are mutually independent. For $u\in(0,1)$, the variable~$G_u$ is a geometric variable
(number of failures until first success, supported in $\{0,1,2,\ldots\}$) with probability of
success $1-u$, i.e., $\mathbb{P}(G_u=n) = u^n(1-u)$, for each $n\ge 0$.

We may write
\[
\ln(P_k(z)) = \sum_{j\ge 1} \ln\frac{1}{1-z^{j^k}} = \sum_{n\ge 1} \frac{\delta_k(n)}{n}\, z^n := g_k(z), \qquad \text{for } |z|<1\,,
\]
where $\delta_k(n) = \sum_{j^k\mid n} j^k$, the sum of the perfect $k$-th powers $j^k$ dividing~$n$. Observe that
the coefficients of the power series~$g_k$ are nonnegative real numbers.

\subsubsection{Fulcrum of $P_k$}\label{subsubsec:fulcrumPk}

Fix an integer $k\ge 1$. The fulcrum $F_k(z)$ of~$P_k$ is given by
\begin{equation}\label{eq:fulcrumPk}
F_k(z) = \ln(P_k(e^z)) = \sum_{j\ge 1} \ln\frac{1}{1-e^{j^k z}} = g_k(e^z), \qquad \text{for } z \text{ such that } \Re(z)<0\,.
\end{equation}
In particular,
\begin{equation}\label{eq:fulcrumPkreal}
F_k(s) = \sum_{j\ge 1} \ln\frac{1}{1-e^{j^k s}}\,, \qquad \text{for } s<0\,.
\end{equation}

The first and second derivatives of~$F_k$ evaluated at $-s$ with $s>0$ admit the expressions
\[
F_k'(-s) = \sum_{j\ge 1} \frac{j^k e^{-j^k s}}{1-e^{-j^k s}}\,, \qquad
F_k''(-s) = \sum_{j\ge 1} \frac{j^{2k} e^{-j^k s}}{(1-e^{-j^k s})^2}\,, \qquad \text{for } s>0\,.
\]

\subsubsection{Gaussianity of $P_k$}\label{subsubsec:GaussianityPk}

We now turn to the asymptotic behaviour of the derivatives
of the fulcrum~$F_k$.
\smallskip

The following lemma is~\cite[Proposition~5.2]{MaciaGaussianity}. It provides the precise asymptotic rate
of growth $F_k^{(m)}(-s)$ as $s\downarrow 0$.

\begin{lemmaL}\label{lem:asympt derivatives Fk}
Fix an integer $k\ge 1$. For any integer $m\ge 0$ we have
\[
F_k^{(m)}(-s) \sim \frac{1}{k}\,\zeta(1+1/k)\,\Gamma(m+1/k)\,\frac{1}{s^{m+1/k}}\,, \qquad \text{as } s\downarrow 0\,.
\]
\end{lemmaL}

\begin{proof}[Sketch of proof]
For $m=0$, we have, approximating series by integrals,
\[
\lim_{s\downarrow 0} s\, F_k(-s^k) = \lim_{s\downarrow 0} \sum_{j=1}^{\infty} s\, \ln\frac{1}{1-e^{-(js)^k}} = \int_0^{\infty} \ln\frac{1}{1-e^{-x^k}}\, dx = \frac{1}{k}\,\zeta(1+1/k)\,\Gamma(1/k).
\]
For $m=1$, we have that
\[
\lim_{s\downarrow 0} s^{k+1} F_k'(-s^k) = \lim_{s\downarrow 0} \sum_{j=1}^{\infty} s\, \frac{(js)^k\, e^{-(js)^k}}{1-e^{-(js)^k}} = \int_0^{\infty} \frac{x^k e^{-x^k}}{1-e^{-x^k}}\, dx = \frac{1}{k}\,\zeta(1+1/k)\,\Gamma(1+1/k).
\]
In general, for $m\ge 2$, let $h(x) = 1/(1-e^{-x})$ for $x>0$. Observe that
\[
\lim_{s\downarrow 0} s^{mk+1} F_k^{(m)}(-s^k) = (-1)^{m-1} \int_0^{\infty} x^{mk}\, h^{(m-1)}(x^k)\, dx := (-1)^{m-1} L\,.
\]
Via a change of variables, $y=x^k$, and successive integration by parts we get that
\[
L = \frac{(-1)^{m-1}}{k}\, \frac{\Gamma(m+1/k)}{\Gamma(1+1/k)} \int_0^{\infty} y^{1/k}\, h(y)\, dy
\]
and since
\[
\int_0^{\infty} y^{1/k}\, h(y)\, dy = \Gamma(1+1/k)\,\zeta(1+1/k)
\]
the result follows.
\end{proof}

Denote the positive constant factors appearing in the asymptotic formulas of Lemma~\ref{lem:asympt derivatives Fk} as
\begin{equation}\label{eq:omegakm}
\omega_{k,m} = \frac{1}{k}\,\zeta(1+1/k)\,\Gamma(m+1/k), \qquad \text{for } k\ge 1 \text{ and } m\ge 0\,.
\end{equation}

These constants, which appear quite frequently in a number of calculations below, differ in the argument of the $\Gamma$ function. Using the functional equation of
the $\Gamma$ function, $\Gamma(z+1)=z\Gamma(z)$, we can write each of them in terms of any other.
Denoting $\omega_{k,1}=\Omega_k$,  we have for each $k\ge 1$ that
\begin{equation}\label{eq:omegarelations}
\omega_{k,0} = k\,\Omega_k, \qquad \omega_{k,1} = \Omega_k, \qquad \text{and} \qquad \omega_{k,2} = (1+1/k)\,\Omega_k\,.
\end{equation}
The actual values of these constants are not used in what follows, just that they exist and that they are positive.
\smallskip

The mean and variance functions of~$P_k$ shall be denoted by~$m_k$ and~$\sigma_k^2$, respectively.
From Lemma~\ref{lem:asympt derivatives Fk} we have that
\begin{equation}\label{eq:mkasympt}
\begin{aligned}
m_k(e^{-s}) &\sim \omega_{k,1}\,\frac{1}{s^{1+1/k}} := \tilde{m}_k(e^{-s}),\\[4pt]
\sigma_k^2(e^{-s}) &\sim \omega_{k,2}\,\frac{1}{s^{2+1/k}} := \tilde{\sigma}_k^2(e^{-s}),
\end{aligned}
\qquad \text{as } s\downarrow 0\,.
\end{equation}

Notice that both $m_k(t)$ and $\sigma_k(t)$ tend to $\infty$ as $t\uparrow 1$.

\begin{corollary}\label{cor:gaussianity of Pk}
The ordinary generating function~$P_k$ of the partitions into $k$-th powers
is Gaussian.
\end{corollary}

We include the argument for completeness; see~\cite{MaciaGaussianity} for further details. \begin{proof} By
Lemma~\ref{lem:asympt derivatives Fk}, for each fixed $m\ge 3$, we have that
\[
\frac{F_k^{(m)}(-s)}{F_k''(-s)^{m/2}} \sim \frac{\omega_{k,m}}{\omega_{k,2}^{m/2}}\, s^{(m/2-1)/k}, \qquad s\downarrow 0.
\]
Since $m\ge 3$, the exponent $(m/2-1)/k$ is positive, and therefore the above ratio tends
to~$0$, as $s\downarrow 0$. The Gaussianity criterion of Theorem~\ref{thm:gaussianty criterion} then applies, and we conclude
that~$P_k$ is Gaussian.
\end{proof}

\begin{remark}\textit{Gaussianity of $P_k$.}\label{rem:altGaussianityPk}\rm\,
To prove the Gaussianity of~$P_k$, we may instead appeal to Theorem~3.2 of~\cite{CFMK}
and verify that
\begin{equation}\label{eq:altcriterion}
\lim_{s\downarrow 0} \frac{\sup_{|\theta|\le A} |F_k'''(-s+\imath \theta)|}{F_k''(-s)^{3/2}} = 0\,, \qquad \text{for every } A>0\,.
\end{equation}
In contrast to the Gaussianity criterion of~\cite{MaciaGaussianity} which we have used above, this criterion
only invokes the third derivative, and not all derivatives of order $\ge 3$, but it involves values
of that third derivative in the whole left half-plane and not just in the negative real axis.
In fact, we have that
\begin{equation}\label{eq:F3bound}
|F_k'''(-s+\imath \theta)| \le F_k'''(-s), \qquad \text{for any } s>0 \text{ and } \theta\in\mathbb{R}\,,
\end{equation}
and thus~\eqref{eq:altcriterion} follows from $\lim_{s\downarrow 0} F_k'''(-s)/F_k''(-s)^{3/2}=0$, which we have checked above
as part of the proof of Corollary~\ref{cor:gaussianity of Pk}.
\smallskip

To verify~\eqref{eq:F3bound}, recall, from~\eqref{eq:fulcrumPk}, that
\[
F_k(z) = g_k(e^z), \qquad \text{for any } z \text{ with } \Re(z)<0\,,
\]
where $g_k$ is a power series in the unit disk with nonnegative coefficients. The nonnegativity of the coefficients of~$g_k$ implies that
\[
|g_k^{(j)}(e^{-s+\imath \theta})| \le g_k^{(j)}(e^{-s}), \qquad \text{for } s>0 \text{ and } \theta\in\mathbb{R} \text{ and } j\ge 0\,.
\]
Since
\[
F_k'''(z) = e^z g_k'(e^z) + 3 e^{2z} g_k''(e^z) + e^{3z} g_k'''(e^z),
\]
we deduce that
\begin{align*}
|F_k'''(-s+\imath \theta)| &\le e^{-s} g_k'(e^{-s}) + 3 e^{-2s} g_k''(e^{-s}) + e^{-3s} g_k'''(e^{-s})\\
&= F_k'''(-s), \qquad \text{for any } s>0 \text{ and any } \theta\in\mathbb{R}\,.
\end{align*}
\end{remark}

\subsubsection{Partitions into distinct $k$-th powers} We denote the generating function of partitions into distinct $k$-th powers by $Q_k$:
$$Q_k(z)=\prod_{n=1}^\infty (1+z^{n^k})\, , \quad \mbox{for every $z \in \mathbb{D}$}\,.$$
The Taylor series expansion around $z=0$ of $Q_k$ is
$$Q_k(z)=\sum_{n=0}^\infty q_k(n) z^n\, , \quad \mbox{for every $z \in \mathbb{D}$}\,,$$
where, for each $n \ge 1$, the coefficient $q_k(n)$ gives the number of partitions of $n$ into distinct $k$-th powers, and where $q_k(0)=1$.

We denote the Khinchin family associated with $Q_k$ by $(Y_{t}^{[k]})_{t \in [0,1)}$.

\smallskip

The generating functions $P_k$ and $Q_k$ are closely related; in particular:
\begin{equation}\label{eq:QkPkrelation}Q_k(z)=\frac{P_k(z)}{P_k(z^2)}\, , \quad \mbox{for every $z \in \mathbb{D}$}\,,\end{equation}
which follows since $1+z=(1-z^2)/(1-z)$, for every $z \in \mathbb{D}$,
and also
\begin{equation}\label{eq:bound from Q_k to P_k}
\frac{|P_k(z)|}{P_k(|z|)}\le
\frac{|Q_k(z)|}{Q_k(|z|)}\, , \quad \mbox{for every $z \in \mathbb{D}$}\,,\end{equation}
which in turn follows since, by the triangle inequality,
$(1-|z|)/|1-z|\le |1+z|/(1+|z|)$, for every $z \in \mathbb{D}$.

\subsection{Some asymptotic estimates}\label{subsec:asymptest}

We already know, as part of Lemma~\ref{lem:asympt derivatives Fk}, that
\[
\ln(P_k(e^{-s})) \sim \omega_{k,0}\,\frac{1}{s^{1/k}} \qquad \text{and} \qquad m_k(e^{-s}) \sim \omega_{k,1}\,\frac{1}{s^{1+1/k}}\,, \qquad \text{as } s\downarrow 0\,,
\]
but since our goal is to apply Theorem~\ref{thm:baezduarte}, we need more precise asymptotic approximations of~$P_k$ and~$m_k$ which weobtain next and record in Corollary~\ref{cor:Pkasympt} and Corollary~\ref{cor:mkasympt}.
\smallskip

The following estimates use standard arguments; we include the details for completeness.

\begin{lemma}\label{lem:lnPk}
Fix an integer $k\ge 1$. We have
\begin{equation}\label{eq:lnPk}
\ln(P_k(e^{-s})) = \omega_{k,0}\,\frac{1}{s^{1/k}} + \frac{1}{2}\ln(s) - k\ln(\sqrt{2\pi}) + o(1), \qquad \text{as } s\downarrow 0\,.
\end{equation}
\end{lemma}

The level of precision of Lemma~\ref{lem:lnPk} is just what is needed to get, simply by exponentiating~\eqref{eq:lnPk}, the following corollary.

\begin{corollary}\label{cor:Pkasympt}
Fix an integer $k\ge 1$, then
\[
P_k(e^{-s}) \sim \sqrt{\frac{s}{(2\pi)^k}}\;\exp\!\left(\frac{1}{k}\,\zeta(1+1/k)\,\Gamma(1/k)\,\frac{1}{s^{1/k}}\right) \sim \sqrt{\frac{s}{(2\pi)^k}}\;\exp\!\left(\omega_{k,0}\,\frac{1}{s^{1/k}}\right), \quad \text{as } s\downarrow 0\,.
\]
\end{corollary}

In the proof of Lemma~\ref{lem:lnPk} we will resort to Euler--Maclaurin summation of order~$2$ in
the following two formats. Here $B_2(t)=t^2-t+\frac{1}{6}$ is the second Bernoulli polynomial
and $\{x\}$ denotes fractional part.

\medskip

\noindent(a) For $\varphi\in C^2[1,\infty)$ and positive integer~$N$.
\begin{equation}\label{eq:EM1}
\sum_{j=1}^{N} \varphi(j) = \int_1^N \varphi(x)\, dx + \frac{1}{2}\big(\varphi(N)+\varphi(1)\big) + \frac{1}{12}\big(\varphi'(N)-\varphi'(1)\big) - \int_1^N \varphi''(x)\,\frac{B_2(\{x\})}{2}\, dx\,.
\end{equation}

\noindent(b) For $\varphi\in C^2[1,\infty)$ with $\varphi(x)\to 0$, $\varphi'(x)\to 0$ as $x\to\infty$, and $\varphi''$ absolutely
integrable on $[1,\infty)$,
\begin{equation}\label{eq:EM2}
\sum_{j=1}^{\infty} \varphi(j) = \int_1^{\infty} \varphi(x)\, dx + \frac{1}{2}\,\varphi(1) - \frac{1}{12}\,\varphi'(1) - \int_1^{\infty} \varphi''(x)\,\frac{B_2(\{x\})}{2}\, dx\,.
\end{equation}

From~\eqref{eq:EM1} applied to $\varphi(x)=\ln(x)$, one obtains the precise standard expression
\begin{equation}\label{eq:lnNfact}
\ln(N!) = N\ln(N) - N + \frac{1}{2}\ln(N) + 1 - \frac{1}{12} + \int_1^N \frac{B_2(\{x\})}{2x^2}\, dx + \frac{1}{12N}\,.
\end{equation}

Stirling's formula implies that $\ln(N!) - (N\ln(N) - N + \frac{1}{2}\ln(N))$ converges to $\ln(\sqrt{2\pi})$ as
$N\to\infty$. Therefore, letting $N\to\infty$ in~\eqref{eq:lnNfact} gives the identity
\[
\ln(\sqrt{2\pi}) = 1 - \frac{1}{12} + \int_1^{\infty} \frac{B_2(\{x\})}{2x^2}\, dx\,,
\]
or
\begin{equation}\label{eq:Stirlingidentity}
\frac{1}{12} - \int_1^{\infty} \frac{B_2(\{x\})}{2x^2}\, dx = 1 - \ln(\sqrt{2\pi}),
\end{equation}
to be used shortly in the proof of~\eqref{eq:lnPk}.

\begin{proof}[Proof of Lemma~\ref{lem:lnPk}]
Set $\delta = s^{1/k}$ and define, for $x>0$,
\[
h(x) = -\ln(1-e^{-x^k}),
\]
and $\varphi(x) = h(x\delta)$. Then $\ln P_k(e^{-s}) = \sum_{j\ge 1} \varphi(j)$.

Since $\varphi$ and its derivatives decay exponentially towards~$0$ we may apply~\eqref{eq:EM2} and
write
\begin{align}\label{eq: sum_asymptotic_P_k}
\sum_{j\ge 1} h(j\delta) = I + \frac{1}{2}\, h(\delta) + R\,,
\end{align}
where $I = \int_1^{\infty} h(x\delta)\, dx$ and
\begin{equation}\label{eq:remainder}
R = -\frac{1}{12}\,\delta\, h'(\delta) - \delta^2 \int_1^{\infty} h''(x\delta)\,\frac{B_2(\{x\})}{2}\, dx.
\end{equation}

\medskip

For the integral term $I$ of \eqref{eq: sum_asymptotic_P_k} we have, substituting $u=x\delta$, that
\[
I = \frac{1}{\delta}\int_{\delta}^{\infty} h(u)\, du = \frac{1}{\delta}\int_0^{\infty} h(u)\, du - \frac{1}{\delta}\int_0^{\delta} h(u)\, du.
\]
The value of the full integral is
\[
\int_0^{\infty} h(u)\, du = \frac{1}{k}\,\Gamma\!\left(\frac{1}{k}\right)\zeta\!\left(1+\frac{1}{k}\right) = \omega_{k,0}.
\]
For the truncated piece, write $h(u) = -k\ln u + \Delta(u)$ where $\Delta(u) = \ln({u^k}/(1-e^{-u^k}))$
is continuous at~$0$ with $\Delta(0)=0$. Then $\int_0^{\delta} h(u)\, du = -k\delta\ln\delta + k\delta + o(\delta)$, giving
\[
I = \frac{\omega_{k,0}}{\delta} + k\ln\delta - k + o(1).
\]

\medskip

For the boundary term of \eqref{eq: sum_asymptotic_P_k} we have, since $h(\delta) = -k\ln\delta + \Delta(\delta)$, that 
\[
\frac{1}{2}\, h(\delta) = -\frac{k}{2}\ln\delta + o(1)\,,
\]
while for the remainder $R$ of \eqref{eq: sum_asymptotic_P_k} we have, using $h'(u) = -\frac{k u^{k-1}}{e^{u^k}-1}$, that
\[
\delta\, h'(\delta) = -\frac{k\delta^k}{e^{\delta^k}-1} = -k + o(1),
\]
and thus
\[
-\frac{1}{12}\,\delta\, h'(\delta) = \frac{k}{12} + o(1).
\]

For the integral in~\eqref{eq:remainder}: since $|h''(u)|\le C/u^2$ for all $u>0$ (the singularity at~$0$ is
exactly of order~$u^{-2}$, and the decay at~$\infty$ is exponential), we have $|h''(x\delta)\delta^2| \le C/x^2$,
which is integrable on $[1,\infty)$ independently of~$\delta$. Since $h''(x\delta)\delta^2 \to k/x^2$ pointwise as
$\delta\downarrow 0$, dominated convergence gives
\[
\delta^2 \int_1^{\infty} h''(x\delta)\,\frac{B_2(\{x\})}{2}\, dx \longrightarrow k\int_1^{\infty} \frac{B_2(\{x\})}{2x^2}\, dx, \qquad \text{as } \delta\downarrow 0\,.
\]

Combining with the derivative correction and appealing to~\eqref{eq:Stirlingidentity}, we get that
\[
R \longrightarrow \frac{k}{12} - k\int_1^{\infty} \frac{B_2(\{x\})}{2x^2}\, dx = k(1-\ln(\sqrt{2\pi})) \qquad \text{as } \delta\downarrow 0\,.
\]

\medskip

In summary: using $\delta=s^{1/k}$ and $\ln\delta = \frac{1}{k}\ln s$:
\begin{align*}
\ln P_k(e^{-s}) &= \underbrace{\frac{\omega_{k,0}}{\delta} + (k\ln\delta - k)}_{I} \underbrace{{}-\frac{k}{2}\ln\delta}_{\frac{1}{2}h(\delta)} + \underbrace{k - k\ln\sqrt{2\pi}}_{R} + o(1)\\
&= \frac{\omega_{k,0}}{s^{1/k}} + \frac{k}{2}\ln\delta - k\ln(\sqrt{2\pi}) + o(1) = \frac{\omega_{k,0}}{s^{1/k}} + \frac{1}{2}\ln s - k\ln(\sqrt{2\pi}) + o(1). \qedhere
\end{align*}
\end{proof}

For the mean $m_k(e^{-s})$ we just need the simple extra precision recorded in the following
lemma.

\begin{lemma}\label{lem:mk}
Fix an integer $k\ge 1$. Then
\begin{equation}\label{eq:mk}
m_k(e^{-s}) = \omega_{k,1}\,\frac{1}{s^{1+1/k}} + O\!\left(\frac{1}{s}\right), \qquad \text{as } s\downarrow 0\,.
\end{equation}
\end{lemma}

\begin{proof}
Define the function~$\varphi(x)$ for $x\ge 0$ by $\varphi(x) = x^k/(e^{x^k}-1)$ for $x>0$ and
$\varphi(0)=1$. The function~$\varphi$ is continuous in $[0,\infty)$ and decreases monotonically from
$1=\varphi(0)$ to $0=\lim_{x\to\infty}\varphi(x)$, since $\psi(x)=x/(e^x-1)$ is monotonically decreasing, $x^k$
is monotonically increasing and $\varphi(x)=\psi(x^k)$.
\smallskip

From monotonicity and since $\varphi$ is bounded above by~$1$, we have that
\[
\sum_{j=1}^{\infty} s\,\varphi(sj) \le \int_0^{\infty} \frac{x^k}{e^{x^k}-1}\, dx = \omega_{k,1} \le \sum_{j=1}^{\infty} s\,\varphi(sj) + s.
\]
This, in terms of~$m_k$, simply says that
\[
0 \le \omega_{k,1} - s^{k+1} m_k(e^{-s^k}) \le s\,, \qquad \text{for any } s>0\,.
\]
Replacing now $s$ by $s^{1/k}$ and dividing by $s^{(k+1)/k}$, the above inequality becomes
\[
0 \le \omega_{k,1}\,\frac{1}{s^{1+1/k}} - m_k(e^{-s}) \le s\,, \qquad \text{for any } s>0\,,
\]
which implies the statement of the lemma.
\end{proof}

As a corollary of Lemma~\ref{lem:mk} we have:

\begin{corollary}\label{cor:mkasympt}
Fix an integer $k\ge 1$ and define
\[
\tilde{m}_k(e^{-s}) = \omega_{k,1}\,\frac{1}{s^{1+1/k}}\,, \qquad \text{for any } s>0.
\]
Then
\[
\frac{m_k(e^{-s})-\tilde{m}_k(e^{-s})}{\sigma_k(e^{-s})} = O(s^{1/(2k)}), \qquad \text{as } s\downarrow 0\,.
\]
\end{corollary}

\begin{proof}
Lemma~\ref{lem:mk} gives that
\[
m_k(e^{-s}) - \tilde{m}_k(e^{-s}) = O(1/s), \qquad \text{as } s\downarrow 0\,.
\]
From~\eqref{eq:mkasympt} we have that
\[
\sigma_k(e^{-s}) \sim \omega_{k,2}^{1/2}\,\frac{1}{s^{1+1/(2k)}}\,, \qquad \text{as } s\downarrow 0\,.
\]
Therefore,
\[
\frac{m_k(e^{-s})-\tilde{m}_k(e^{-s})}{\sigma_k(e^{-s})} = O(s^{1/(2k)}), \qquad \text{as } s\downarrow 0\,.
\]
This concludes the proof.
\end{proof}

\subsection{Strong Gaussianity of power partitions}\label{subsec:strongGaussianityproof}

In this section we verify
that the Khinchin family $(X^{[k]}_t)_{t\in[0,1)}$ is strongly Gaussian.
\smallskip

For each $t\in(0,1)$, we let $s=-\ln t > 0$ so that $e^{-s}=t$. We write~$s$ rather than~$s(t)$
as no confusion will arise.

\subsubsection{Estimation of a Diophantine sum}

Fix an integer $k \ge 1$. We define $W_k(s,\phi)$  for $s >0$ and $\phi \in \R$ as the following Diophantine sum
$$W_k(s,\phi)=\sum_{m \in I_s} \| m^k \phi/(2\pi)\|^2\,,
$$
where $I_s$ denotes the interval in $\N$ of those integers $m$ such that
$$\frac{1}{(4s)^{1/k}} <m \le \frac{1}{(2s)^{1/k}}$$
and where, for $x \in \R$, we use $\|x\|$ to denote distance of $x$ to $\mathbb{Z}$:
$$\|x\|=\min\{|x-n|:\,n \in \Z\}\,.$$

The following lower bound of  $W_k(s,\phi)$, due to Tenenbaum, Wu and Li in~\cite{TWLarxiv}, is the key to the confirmation of the strong Gaussianity of $P_k$, and also of $Q_k$ later on. It is stated and verified  within the proof of Lemma~2.3 in~\cite{TWLarxiv}.

\begin{lemmaL}[Tenenbaum,Wu and Li]\label{lema:estimate of W}
There are positive constants $d_1, d_2>0$ depending only on~$k$, such that for every $s>0$,
$$W_k(s,\phi)\ge \begin{cases}d_1 \phi^2/s^{2+1/k}, &\mbox{for $|\phi|\le 2 \pi s$}, \\ \\
d_2 /s^{1/k}, &\mbox{for $2\pi s\le |\phi|\le \pi$}\,.\end{cases}$$
\end{lemmaL}

\subsubsection{Bounds on the characteristic functions of $P_k$ and of $Q_k$}\label{subsubsec:bounds}

We have

\begin{lemma}\label{lem:cotas de G_k en terminos de W} For some universal constant $C>0$, it holds, for each integer $k\ge 1$, that
$$
\ln\frac{Q_k(e^{-s})}{|Q_k(e^{-s -\imath \phi})|}\ge C \, W_k(s,\phi)\, , \quad \mbox{for each $s>0$ and $\phi \in \R$}\,.
$$
\end{lemma}

\begin{proof} We start by observing that
$$ |1+e^{-u+\imath \theta}|^2=(1+e^{-u})^2-4 e^{-u} \sin^2(\theta/2)\,, \quad \mbox{for any $u, \theta\in \R$}\,,$$
here we have used the relation $\sin^2(\theta/2) = (1-\cos(\theta))/2$, which holds for any $\theta \in \R$.
\smallskip

Using that $(1+y)\le e^y$, for every $y \in \R$, we deduce that
\[
\frac{\left|1 + e^{-u - \imath\theta}\right|^2}{(1 + e^{-u})^2}
= 1 - \frac{4e^{-u}\sin^2(\theta/2)}{(1 + e^{-u})^2}
\leq \exp\!\left(-\frac{4e^{-u}\sin^2(\theta/2)}{(1 + e^{-u})^2}\right)
\]
\
Substituting $u = m^k s$ and $\theta = m^k \phi$, we get that

\begin{equation}
\tag{$\dagger$}
\ln \frac{Q_k(e^{-s})^2}{|Q_k(e^{-s - \imath\phi})|^2}
\geq \sum_{m \geq 1} \frac{4e^{-m^k s}\sin^2(m^k \phi/2)}{(1 + e^{-m^k s})^2}
\end{equation}

For $m \in I_s$ we have that $\frac{1}{4} < m^k s \leq \frac{1}{2}$, and therefore
\[
e^{-m^k s} \geq e^{-1/2}\quad \mbox{and} \quad
(1 + e^{-m^k s}) \leq 1 + e^{-1/4}\,,
\]
and so
\[
\frac{4e^{-m^k s}}{(1 + e^{-m^k s})^2}
\geq \frac{4e^{-1/2}}{(1 + e^{-1/4})^2} := \Delta
\]

Since  $\sin^2(\pi x) \geq 4\|x\|^2$, for each  $x \in \R$,
it follows from $(\dagger)$ that
\[
\ln \frac{Q_k(e^{-s})^2}{|Q_k(e^{-s - \imath \phi})|^2}
\geq 4\Delta \cdot W_k(s,\phi)
\]
as stated, with $C=2 \Delta$. \end{proof}

From the key Lemma \ref{lema:estimate of W} combined with Lemma \ref{lem:cotas de G_k en terminos de W} we deduce that

\begin{lemma}\label{lem:bounds of char de Q}
There are positive constants $d_1, d_2>0$ depending only on~$k$ such that
\[
\frac{|Q_k(e^{-s+\imath\phi})|}{Q_k(e^{-s})} \le
\begin{cases}
\exp\!\big(-d_1 \phi^2 s^{-(2+1/k)}\big), & \text{if } |\phi|\le 2\pi s\,,\\[4pt]
\exp\!\big(-d_2 s^{-1/k}\big), & \text{if } 2\pi s < |\phi| \le \pi\,.
\end{cases}
\]
\end{lemma}

From Lemma \ref{lem:bounds of char de Q} and the comparison \eqref{eq:bound from Q_k to P_k} we obtain

\begin{lemmaL}\label{lem:bounds of char de P}
There are positive constants $d_1, d_2>0$ depending only on~$k$ such that
\[
\frac{|P_k(e^{-s+\imath\phi})|}{P_k(e^{-s})} \le
\begin{cases}
\exp\!\big(-d_1 \phi^2 s^{-(2+1/k)}\big), & \text{if } |\phi|\le 2\pi s\,,\\[4pt]
\exp\!\big(-d_2 s^{-1/k}\big), & \text{if } 2\pi s < |\phi| \le \pi\,.
\end{cases}
\]
\end{lemmaL}

This Lemma~\ref{lem:bounds of char de P} is Lemma~2.3 of~\cite{TWLarxiv} of Tenenbaum, Wu and Li,  in the notation of the present  paper.
\smallskip

The argument in~\cite{TWLarxiv} derives Lemma \ref{lem:bounds of char de P}  from the estimate recorded in Lemma \ref{lema:estimate of W}; we have just inserted the estimates for $Q_k$ in between.
\smallskip

The threshold $2\pi s$ may be replaced by any pair of overlapping thresholds $B_2 s < B_1 s$
(with $0<B_2<B_1$), at the cost of adjusting the constants $d_1, d_2$ to $D_1, D_2$ which now
depend on~$k$ and also on $B_1$ and $B_2$, so that
\begin{equation}\label{eq:TWLgeneral}
\frac{|P_k(e^{-s+\imath\phi})|}{P_k(e^{-s})} \le
\begin{cases}
\exp\!\big(-D_1 \phi^2 s^{-(2+1/k)}\big), & \text{if } |\phi|\le B_1 s\,,\\[4pt]
\exp\!\big(-D_2 s^{-1/k}\big), & \text{if } B_2 s < |\phi| \le \pi\,.
\end{cases}
\end{equation}
Notice that in the overlap $B_2 s \le |\phi| \le B_1 s$, one has that $\phi^2 s^{-(2+1/k)}$ is comparable to
$s^{-1/k}$, so both bounds give comparable decay.

For the characteristic function of~$\breve{X}^{[k]}_t$ we have that
\[
\big|\E\!\big(e^{\imath \theta\breve{X}^{[k]}_t}\big)\big| = \frac{\big|P_k(e^{-s+\imath \theta/\sigma_k(e^{-s})})\big|}{P_k(e^{-s})}\,,
\]
since taking the modulus eliminates the phase term.
\smallskip

From~\eqref{eq:mkasympt} we have, for positive constants $a<A$, that
\begin{equation}\label{eq:sigmabounds}
a \le \sigma_k(e^{-s})\, s^{1+(1/(2k))} < A\,, \qquad \text{for } s\in(0,\ln 2).
\end{equation}

The bound on $\big|\E\!\big(e^{\imath \theta\breve{X}^{[k]}_t}\big)\big|$ we are after is recorded in the following corollary.

\begin{corollaryL}\label{cor:bounds for char Pk}
For any constant $C>0$ there are positive constants $c_1$ and $c_2$ depending
only on~$k$ and~$C$ such that for $s\in(0,\ln 2)$ we have that
\[
\big|\E\!\big(e^{\imath \theta\breve{X}^{[k]}_t}\big)\big| \le
\begin{cases}
e^{-c_1\theta^2}, & \text{if } |\theta|\le C\,\dfrac{1}{s^{1/(2k)}}\,,\\[8pt]
e^{-c_2\frac{1}{s^{1/k}}}, & \text{if } |\theta|\ge C\,\dfrac{1}{s^{1/(2k)}}\,.
\end{cases}
\]
\end{corollaryL}

\begin{proof}
We take $B_1 = C/a$ and $B_2 = C/A$ (note that $a<A$ by~\eqref{eq:sigmabounds}, so $B_2<B_1$). Let
$D_1$ and $D_2$ be as in~\eqref{eq:TWLgeneral} and $c_1 = D_1/A^2$ and $c_2 = D_2$.
\smallskip

If $|\theta|\le C s^{-1/(2k)}$, then $|\phi|\le (C/a)s = B_1 s$ and so by~\eqref{eq:TWLgeneral} and~\eqref{eq:sigmabounds} we get that
\[
\big|\E\!\big(e^{\imath \theta\breve{X}^{[k]}_t}\big)\big| \le \exp(-(D_1/A^2)\theta^2) = \exp(-c_1\theta^2).
\]
If $|\theta|\ge C s^{-1/(2k)}$, then $|\phi|\ge (C/A)s = B_2 s$ and so by~\eqref{eq:TWLgeneral} we get that
\[
\big|\E\!\big(e^{\imath \theta\breve{X}^{[k]}_t}\big)\big| \le \exp(-D_2 s^{-1/k}) = \exp(-c_2 s^{-1/k}). \qedhere
\]
\end{proof}

\subsubsection{Strong Gaussianity of~$P_k$}

Now we have at our disposal all the ingredients to verify that $P_k$ is strongly Gaussian. 

\begin{theorem}\label{thm:strongGaussianity}
The Khinchin family $(X^{[k]}_t)_{t\in[0,1)}$ associated with the generating function
$P_k(z)$ of partitions into $k$-th powers is strongly Gaussian.
\end{theorem}

\begin{proof}
Corollary~\ref{cor:gaussianity of Pk} tells us that the family $(X^{[k]}_t)_{t\in[0,1)}$ is Gaussian. Also from~\eqref{eq:mkasympt}, we
see that $\lim_{s\downarrow 0}\sigma_k(e^{-s})=+\infty$.
\smallskip

Now, for $|\theta|\le C/s^{1/(2k)}$ and $s\in(0,\ln 2)$ (or $t\in(1/2,1)$), we have from Corollary~\ref{cor:bounds for char Pk}
that
\[
\big|\E\!\big(e^{\imath \theta\breve{X}^{[k]}_t}\big) - e^{-\theta^2/2}\big| \le e^{-c_1\theta^2} + e^{-\theta^2/2},
\]
and thus Gaussianity and the Dominated Convergence Theorem give that
\begin{equation}\label{eq:SG1}
\lim_{s\downarrow 0} \int_{|\theta|\le C/s^{1/(2k)}} \big|\E\!\big(e^{\imath \theta\breve{X}^{[k]}_t}\big) - e^{-\theta^2/2}\big|\, d\theta = 0\,.
\end{equation}

We also have that
\begin{equation}\label{eq:SG2}
\lim_{s\downarrow 0} \int_{|\theta|\ge C/s^{1/(2k)}} e^{-\theta^2/2}\, d\theta = 0\,.
\end{equation}

From Corollary~\ref{cor:bounds for char Pk}, for $s\in(0,\ln 2)$, we get the bound
\[
\int_{\pi\sigma_k(e^{-s})\ge\theta\ge C/s^{1/(2k)}} \big|\E\!\big(e^{\imath \theta\breve{X}^{[k]}_t}\big)\big|\, d\theta \le \pi\sigma_k(e^{-s})\, e^{-c_2/s^{1/k}}.
\]
Taking into account~\eqref{eq:sigmabounds}, we see that $\lim_{s\downarrow 0}\sigma_k(e^{-s})\, e^{-c_2/s^{1/k}}=0$, since exponential decay
beats polynomial growth. Therefore, we have that
\begin{equation}\label{eq:SG3}
\lim_{s\downarrow 0} \int_{\pi\sigma_k(e^{-s})\ge\theta\ge C/s^{1/(2k)}} \big|\E\!\big(e^{\imath \theta\breve{X}^{[k]}_t}\big)\big|\, d\theta = 0\,.
\end{equation}

From~\eqref{eq:SG2} and~\eqref{eq:SG3} we conclude that
\begin{equation}\label{eq:SG4}
\lim_{s\downarrow 0} \int_{\pi\sigma_k(e^{-s})\ge|\theta|\ge C/s^{1/(2k)}} \big|\E\!\big(e^{\imath \theta\breve{X}^{[k]}_t}\big) - e^{-\theta^2/2}\big|\, d\theta = 0\,.
\end{equation}

Finally, the combination of~\eqref{eq:SG1} and~\eqref{eq:SG4} gives the strong Gaussianity of $P_k(z)$.
\end{proof}

\section{Asymptotic formula of power partitions}\label{sec:asymptotic}

We have seen in Theorem~\ref{thm:strongGaussianity} that~$P_k$ is strongly Gaussian, and thus, given the approximation of~$m_k$ from Corollary~\ref{cor:mkasympt},  we are ready to apply the version of Hayman's asymptotic
formula registered in Theorem~\ref{thm:baezduarte} to obtain the asymptotic formula Hardy-Ramanujan and Wright for partitions into
$k$-th powers.

\begin{theoremL}\label{thm:asymptotic of pk(n)}
Fix an integer $k\ge 1$, then
\[
p_k(n) \sim \frac{1}{(2\pi)^{(k+1)/2}} \cdot \frac{\Omega_k^{k/(k+1)}}{(1+1/k)^{1/2}} \cdot \frac{1}{n^{(3k+1)/(2k+2)}}\, \exp\!\Big((k+1)\,\Omega_k^{k/(k+1)}\, n^{1/(k+1)}\Big), \quad \text{as } n\to\infty,
\]
where $\Omega_k = \frac{1}{k}\,\zeta(1+1/k)\,\Gamma(1+1/k)$.
\end{theoremL}

Observe that according to the statement of Theorem~\ref{thm:asymptotic of pk(n)} the constants $\alpha_k$ and $\beta_k$ of
the Hardy--Ramanujan formula~(HR) are given by
\[
\beta_k = (k+1)\,\Omega_k^{k/(k+1)} \qquad \text{and} \qquad \alpha_k = \frac{\Omega_k^{k/(k+1)}}{(2\pi)^{(k+1)/2}\,(1+1/k)^{1/2}}\,,
\]
as they should. Also, for $k=1$, we have $\Omega_1=\zeta(2)\,\Gamma(2)=\pi^2/6$ and $\alpha_1=1/(4\sqrt{3})$ and
$\beta_1=\pi\sqrt{2/3}$, as in the Hardy--Ramanujan asymptotic formula for general partitions.

\begin{proof}
We are going to apply Theorem~\ref{thm:baezduarte} of Báez-Duarte with the approximation~$\tilde{m}_k$ of~$m_k$ given by
\[
\tilde{m}_k(e^{-s}) = \Omega_k\,\frac{1}{s^{1+1/k}}\,, \qquad \text{for any } s>0\,,
\]
whose use is justified by Corollary~\ref{cor:mkasympt}.

The constants $\omega_{k,m}$ and $\Omega_k$ from~\eqref{eq:omegakm} and~\eqref{eq:omegarelations} will intervene in what follows.

For $n\ge 1$, take $\tau_n = e^{-s_n}$, with $s_n$ given by
\begin{equation}\label{eq:sn}
s_n = (\Omega_k/n)^{k/(k+1)},
\end{equation}
so that $\tilde{m}_k(\tau_n) = \tilde{m}_k(e^{-s_n}) = n$.
\smallskip

Next we just have to plug into the general asymptotic formula~\eqref{eq:BDfinal} the formula for
$\tau_n^{-n}$, and the asymptotics, taking into account~\eqref{eq:omegarelations}, of $P_k(e^{-s_n})$ and of $\sigma_k(e^{-s_n})$ provided
respectively by Corollary~\ref{cor:Pkasympt} and~\eqref{eq:mkasympt}.

Observe that from~\eqref{eq:sn}
\begin{equation}\tag{$\star_1$}
\tau_n^{-n} = \exp\!\Big(\Omega_k^{k/(k+1)}\, n^{1/(k+1)}\Big), \qquad \text{for each } n\ge 1\,.
\end{equation}

Using~\eqref{eq:mkasympt} and~\eqref{eq:sn} we have that
\begin{equation}\tag{$\star_2$}
\frac{1}{\sigma_k(e^{-s_n})} \sim \frac{\Omega_k^{k/(2k+2)}}{(1+1/k)^{1/2}} \cdot \frac{1}{n^{(2k+1)/(2k+2)}}\,, \qquad \text{as } n\to\infty.
\end{equation}

And, finally, using Corollary~\ref{cor:Pkasympt}, \eqref{eq:omegarelations} and~\eqref{eq:sn} we obtain that
\begin{equation}\tag{$\star_3$}
P_k(e^{-s_n}) \sim \frac{1}{(2\pi)^{k/2}}\, \frac{\Omega_k^{k/(2k+2)}}{n^{k/(2k+2)}}\, \exp\!\Big(k\,\Omega_k^{k/(k+1)}\, n^{1/(k+1)}\Big), \qquad \text{as } n\to\infty.
\end{equation}

Substituting $(\star_1)$, $(\star_2)$, and $(\star_3)$ into~\eqref{eq:BDfinal} we obtain the result.
\end{proof}

\section{$P_k$ and the Hayman class}\label{sec:Hayman}

In this section we show that the generating function~$P_k$ of partitions into $k$-th powers
belongs to the Hayman class. This notion of Hayman class was originally introduced
in~\cite{Hayman} and has been studied in depth, from the point of view of Khinchin families, in~\cite{CFMK}.

Power series in the Hayman class are strongly Gaussian; see~\cite[Theorem~3.8]{CFMK}. Thus the
result of this section implies Theorem~\ref{thm:strongGaussianity}.

That the generating function of ordinary partitions~$P$ is in the Hayman class appears
in~\cite[Thm.~6.2]{CFMK}.

A power series $f\in\mathcal{K}$ with radius of convergence~$R$ and Khinchin family $(Z_t)_{t\in[0,R)}$
is said to be in the Hayman class if
\begin{equation}\tag{H0}
\lim_{t\uparrow R} \sigma_f(t) = +\infty\,,
\end{equation}
and besides, there exists a function $h\colon [0,R)\to(0,\pi]$, termed the cut, such that the
following two conditions are satisfied
\begin{equation}\tag{H1}
\lim_{t\uparrow R} \sup_{|\theta|\le h(t)\sigma_f(t)} \left|\E\!\left(e^{\imath\theta\breve{Z}_t}\right) e^{\theta^2/2} - 1\right| = 0, \qquad \text{(major arc)};
\end{equation}
and
\begin{equation}\tag{H2}
\lim_{t\uparrow R} \sigma_f(t) \sup_{h(t)\sigma_f(t)<|\theta|\le\pi\sigma_f(t)} \big|\E\!\big(e^{\imath\theta\breve{Z}_t}\big)\big| = 0. \qquad \text{(minor arc)}.
\end{equation}

\smallskip

The next Lemma~\ref{lem:majorarc} provides us with a condition in terms of the fulcrum which implies the
major arc condition.

\begin{lemma}\label{lem:majorarc}
For $f\in\mathcal{K}$ \emph{non-vanishing} in $\mathbb{D}(0,R)$ with Khinchin family $(Z_t)_{t\in[0,R)}$ and
with fulcrum $H(z) = \ln f(e^z)$, for $\Re(z)<\ln R$, and a cut function~$h(t)$, if
\begin{equation}\label{eq:majorarccond}
\lim_{t\uparrow R} \sup_{\varphi\in\mathbb{R}} |H'''(-s+\imath\varphi)| \cdot h(t)^3 = 0\,,
\end{equation}
where $t=e^{-s}$, for $s>\ln(1/R)$, then condition \textup{(H1)} holds.
\end{lemma}

\begin{proof}
For such $f$ it holds that
\begin{equation}\label{eq:logcharbound}
\left|\ln \E\!\left(e^{\imath\theta\breve{Z}_t}\right) + \frac{\theta^2}{2}\right| \le \frac{\sup_{\varphi\in\mathbb{R}} |H'''(-s+\imath\varphi)|}{\sigma_f(t)^3} \cdot \frac{|\theta|^3}{6}\,, \qquad t=e^{-s},
\end{equation}
see~\cite[Theorem~3.2 and~(4.3)]{CFMK}.

Let us denote
\[
A(t) = \sup_{\varphi\in\mathbb{R}} |H'''(-s+\imath\varphi)| \cdot \frac{h(t)^3}{6}
\]
with $t=e^{-s}$ and $t\in(0,R)$; by hypothesis~\eqref{eq:majorarccond} we have $\lim_{t\uparrow R}A(t)=0$.
\smallskip

From~\eqref{eq:logcharbound} we deduce, for $|\theta|\le h(t)\sigma_f(t)$, that
\[
\left|\ln \E\!\left(e^{\imath \theta\breve{Z}_t}\right) + \frac{\theta^2}{2}\right| \le A(t),
\]
which combined with the numerical inequality $|e^w-1|\le |w|e^{|w|}$, which holds for every
$w\in\mathbb{C}$, gives that
\[
\sup_{|\theta|\le h(t)\sigma_f(t)} \left|\E\!\left(e^{\imath \theta\breve{Z}_t}\right) e^{\theta^2/2} - 1\right| \le A(t)\, e^{A(t)},
\]
which implies (H1), since $A(t)\to 0$, as $t\uparrow R$.
\end{proof}

Next, we turn to the function~$P_k$.

\begin{theorem}\label{thm:Hayman}
The generating function $P_k(z)$ of partitions into $k$-th powers belongs to
the Hayman class.
\end{theorem}

\begin{proof}
We write $t=e^{-s}$ with $s\downarrow 0$ throughout.

Condition (H0) for~$P_k$ holds since from~\eqref{eq:mkasympt}, $\lim_{s\downarrow 0}\sigma_k(e^{-s})=+\infty$.

Next we define an appropriate cut function~$h$. We fix any $\alpha$ in the (nonempty) interval
\begin{equation}\label{eq:alphainterval}
\big(1+1/(3k),\; 1+1/(2k)\big),
\end{equation}
and set $h(e^{-s})=s^{\alpha}$.
\smallskip

We next verify that with this cut~$h$, conditions (H1) and (H2) for the Khinchin family
of~$P_k$ to be in the Hayman class are satisfied. The lower bound  $\alpha > 1+1/(3k)$ is required
to check (H1), while the upper bound  $\alpha < 1+1/(2k)$ is needed to verify (H2).

\medskip

\noindent\textit{Condition \textup{(H1)}.} Equation~\eqref{eq:F3bound} gives
\begin{equation}\label{eq:F3boundbis}
\sup_{\varphi\in\mathbb{R}} \big|F_k'''(-s+\imath\varphi)\big| \le F_k'''(-s), \qquad \text{for } s>0,
\end{equation}
while Lemma~\ref{lem:asympt derivatives Fk} gives that
\[
F_k'''(-s) \sim \omega_{k,3}\, s^{-(3+1/k)}, \qquad \text{as } s\downarrow 0,
\]
and thus since $h(e^{-s})=s^{\alpha}$,
\begin{equation}\label{eq:F3boundtris}
\sup_{\varphi\in\mathbb{R}} \big|F_k'''(-s+\imath\varphi)\big| \cdot h(e^{-s})^3 \le F_k'''(-s)\cdot s^{3\alpha} \sim \omega_{k,3}\, s^{3\alpha-(3+1/k)} \to 0, \qquad \text{as } s\downarrow 0,
\end{equation}
since $\alpha > 1+1/(3k)$. Lemma~\ref{lem:majorarc} then implies that condition (H1) holds for~$P_k$.

\medskip

\noindent\textit{Condition \textup{(H2)}.} The verification of this condition is based upon Corollary~\ref{cor:bounds for char Pk}.

From~\eqref{eq:mkasympt} and the definition of~$h$ we have that
\begin{equation}\label{eq:sigmahasympt}
\begin{aligned}
\sigma_k(e^{-s}) &\sim \frac{\omega_{k,2}^{1/2}}{s^{1+(1/2k)}}\,,\\[4pt]
h(e^{-s})\,\sigma_k(e^{-s}) &\sim \frac{\omega_{k,2}^{1/2}}{s^{1+(1/2k)-\alpha}}\,,
\end{aligned}
\qquad \text{as } s\downarrow 0\,.
\end{equation}

Let $C, c_1, c_2$ be the constants appearing in Corollary~\ref{cor:bounds for char Pk}. Define the following two
sub-arcs of the minor arc
\begin{align*}
I_L(s) &= \big[h(e^{-s})\sigma_k(e^{-s}),\; C s^{-1/(2k)}\big),\\
I_R(s) &= \big[C s^{-1/(2k)},\; \pi\sigma_k(e^{-s})\big).
\end{align*}
For $s$ sufficiently small, these two sub-arcs partition the minor arc, since from~\eqref{eq:sigmahasympt} it
follows that
\begin{equation}\label{eq:arcpartition}
h(e^{-s})\,\sigma_k(e^{-s}) < C s^{-1/(2k)} < \pi\,\sigma_k(e^{-s}), \qquad \text{for  } s>0 \text{ small enough}.
\end{equation}

For $|\theta|\in I_L(s)$, since $|\theta|\ge h(e^{-s})\sigma_k(e^{-s})$, Corollary~\ref{cor:bounds for char Pk} gives that
\begin{equation}\tag{$\flat_L$}
|\E(e^{\imath\theta\breve{X}^{[k]}_t})| \le e^{-c_1(h(e^{-s})\sigma_k(e^{-s}))^2}.
\end{equation}

For $|\theta|\in I_R(s)$, Corollary~\ref{cor:bounds for char Pk} gives directly that
\begin{equation}\tag{$\flat_R$}
|\E(e^{\imath\theta\breve{X}^{[k]}_t})| \le e^{-c_2/s^{1/k}}.
\end{equation}

Let $\beta = \min\{2+1/k-2\alpha,\; 1/k\} > 0$. Appealing to~\eqref{eq:sigmahasympt}, we may combine $(\flat_L)$ and
$(\flat_R)$ and obtain, for a certain constant $c_3>0$ and for $s$ small enough, that
\[
\sup_{h(t)\sigma_k(t)<|\theta|\le\pi\sigma_k(t)} \big|\E\!\big(e^{\imath\theta\breve{X}_t}\big)\big| \le e^{-c_3/s^{\beta}}.
\]
This bound combined with the asymptotics for~$\sigma_k$ of~\eqref{eq:sigmahasympt} gives that (H2) holds.
\end{proof}

As mentioned above, for $k=1$, that is, for the ordinary partition function~$P$, this
result appears in~\cite[Thm.~6.2]{CFMK}. The proof above extends the argument in~\cite{CFMK} to all
$k\ge 1$, with Lemmas~\ref{lem:asympt derivatives Fk} and~\ref{lem:bounds of char de P} replacing the corresponding estimates therein.

\section{Partitions into distinct $k$-th powers}\label{sec:distinct}

In this section we show  that $Q_k$, the generating function of partitions into \emph{distinct} $k$-th powers, is, like $P_k$, a Hayman function.

\medskip

From the relation \eqref{eq:QkPkrelation} between $Q_k$ and $P_k$, we will derive that the variance of the family of $Q_k$ goes to infinity and, by appealing to Lemma \ref{lem:majorarc},  that condition (H1), the major arc condition, holds, with the same cut function as for $P_k$.

That condition (H2) holds for $Q_k$ will be derived from the key estimate of Lemma \ref{lema:estimate of W}.

\medskip

As for notation: we denote the fulcrum of~$Q_k$ by~$G_k$, thus $$G_k(z)=\ln(Q_k(e^z))\,, \quad \mbox{for $\Re(z)<0$},$$ and the mean and variance functions of~$Q_k$ are denoted by~$\mu_k(t)$ and~$\eta_k^2(t)$, respectively, for $t \in (0,1)$.

\subsection{Relation between $Q_k$ and $P_k$}
The treatment of \(Q_k\) runs parallel to that of \(P_k\), and we exploit
this throughout. The identity $Q_k(z)={P_k(z)}/{P_k(z^2)}$
lets us read off the fulcrum, the mean and the variance of \(Q_k\) from
those of \(P_k\), while the comparison~\eqref{eq:bound from Q_k to P_k}
transfers the characteristic-function estimates in the opposite
direction. We record these relations before verifying the Hayman
conditions.

\subsubsection{Fulcrum and Gaussianity of $Q_k$.}

From the relation \eqref{eq:QkPkrelation} between $P_k$ and $Q_k$, we deduce that the respective fulcrums $F_k(z)$ and $G_k(z)$ are related by
\begin{equation}\label{eq:QkPkFulcrumrelation}
G_k(z) = F_k(z) - F_k(2z), \qquad \mbox{for each $z\in\C$ such that $\Re(z)<0$}.
\end{equation}

\smallskip

For the derivatives $G_k^{(m)}$ of $G_k$ we deduce from  \eqref{eq:QkPkFulcrumrelation} that
\begin{equation}\label{eq:QkPkFulcrumrelation derivatives}
G_k^{(m)}(s) = F_k^{(m)}(s) - 2^m F_k^{(m)}(2s), \qquad \text{for } s<0 \text{ and } m\ge 1\,,
\end{equation}
and thus, for each $m\ge 1$, the asymptotic formula for~$F_k^{(m)}$ of Lemma~\ref{lem:asympt derivatives Fk} translates into
the following asymptotic formula for~$G_k^{(m)}$:
\begin{equation}\label{eq:Gkasympt}
G_k^{(m)}(-s) \sim (1-2^{-1/k})\,\frac{1}{k}\,\zeta(1+1/k)\,\Gamma(m+1/k)\,\frac{1}{s^{m+1/k}}\,, \qquad \text{as } s\downarrow 0\,.
\end{equation}

From the asymptotic formula~\eqref{eq:Gkasympt} we deduce that, for any $m\ge 3$, there exists a
constant $C_{k,m}>0$ such that
\[
\frac{G_k^{(m)}(-s)}{G_k''(-s)^{m/2}} \sim C_{k,m}\, s^{(m/2-1)/k}, \qquad \text{as } s\downarrow 0\,.
\]
As in the proof of Corollary~\ref{cor:gaussianity of Pk} this implies, via the Gaussianity criterion of Theorem~\ref{thm:gaussianty criterion},
that \emph{$Q_k(z)$ is Gaussian.}

\subsubsection{Asymptotic of variance $\eta_k^2(e^{-s})$ of $Q_k$.}\label{seccion:comparability of variances}

From \eqref{eq:Gkasympt} and Lemma \ref{lem:asympt derivatives Fk}, observe that for derivatives of order $m \ge 2$,
\begin{equation}\label{eq:derivativesPkQk}
\frac{G_k^{(m)}(-s)}{F_k^{(m)}(-s)}\to 1-{2^{-1/k}}\, , \quad \mbox{as $s \downarrow 0$}\,.
\end{equation}

Thus the variances of the Khinchin families of $P_k$ and of $Q_k$ grow to $+\infty$ at the same rate as $t \uparrow 1$, and, in fact,
$$\lim_{s \downarrow 0} \frac{\eta^2_k(e^{-s})}{\sigma^2_k(e^{-s})}=1-{2^{-1/k}}\,,$$
and, because of \eqref{eq:mkasympt}, we also have that
\begin{equation}\label{eq:asymptotic of etak}\eta^2_k(e^{-s})\sim (1-{2^{-1/k}}) \, \omega_{k,2}\,\frac{1}{s^{2+1/k}}\, , \quad \mbox{as $s \downarrow 0$}\,.\end{equation}

In particular, $\lim_{t\uparrow 1} \eta_k(t)=+\infty$.

\begin{remark}\textit{Gaussianity of $Q_k$.}\rm \, In Remark~\ref{rem:altGaussianityPk} we gave an alternative argument to show that $P_k$ is Gaussian, based on the fact that $P_k=\exp(g_k)$ where $g_k$ has nonnegative Taylor coefficients.

For the generating function~$Q_k$ of partitions into distinct $k$-th powers:
\[
Q_k(z) = \prod_{j=1}^{\infty} (1+z^{j^k}),
\]
we may write analogously that $Q_k \equiv \exp(h_k)$, where $h_k$ is the power series
\[
h_k(z) = \sum_{n=1}^{\infty} \frac{\epsilon_k(n)}{n}\, z^n \qquad \text{where} \qquad \epsilon_k(n) = -\sum_{j^k\mid n} (-1)^{n/j^k} j^k\,.
\]
For $k=1$, the coefficients of~$h_k$ are nonnegative: they can be expressed as $\epsilon_1(n) = \sum_{\substack{j\mid n;\\ j\;\text{odd}}} j$. But for $k\ge 2$, the coefficients $\epsilon_k(n)$ change sign infinitely often as $n$ grows.

\smallskip

In order to show that~$Q_k$ is Gaussian, the argument in Remark~\ref{rem:altGaussianityPk} appealing to Theorem~3.2 of~\cite{CFMK} is of no avail, and we have to resort to the Gaussianity criterion involving all of the
derivatives, as we have discussed above.

\end{remark}

\subsubsection{Asymptotic estimates for $Q_k(e^{-s})$.}

Fix $k \geq 1$. Combining the identity \eqref{eq:QkPkFulcrumrelation} connecting $Q_k$ and $P_k$,
with Lemma \ref{lem:lnPk}, which contains the asymptotic of $\ln (P_k(e^{-s}))$, we obtain that
\begin{lemma}\label{lem: ln(Q_k)_asymptotic} For each $k \geq 1$, it holds that
\begin{align*}
	\ln(Q_k(e^{-s})) = (1-2^{-1/k})\,\omega_{k,0}\, s^{-1/k}-\ln(\sqrt{2})+o(1)\,, \quad \text{ as } s \downarrow 0\,.
\end{align*}
\end{lemma}
\begin{proof} Lemma \ref{lem:lnPk} gives that
\begin{equation*}
	\ln(P_k(e^{-s})) = \omega_{k,0}\,\frac{1}{s^{1/k}} + \frac{1}{2}\ln(s) - k\ln(\sqrt{2\pi}) + o(1), \qquad \text{as } s\downarrow 0\,,
\end{equation*} and from \eqref{eq:QkPkrelation} we deduce as stated that
\begin{align*}
\ln(Q_k(e^{-s})) = (1-2^{-1/k})\,\omega_{k,0}\, s^{-1/k}-\ln(\sqrt{2})+o(1)\,, \quad \text{ as } s \downarrow 0\,.
\end{align*}
\end{proof}
\newpage

As a direct corollary we have~\begin{corollary}\label{cor: ln(Q_k)_asymptotic} For each integer $k \ge 1$
\begin{align*}
Q_k(e^{-s}) =\frac{1}{\sqrt{2}}\,\exp\!\left(\omega_{k,0}\,(1-2^{-1/k})\,s^{-1/k}\right)\bigl(1+o(1)\bigr), \qquad s\downarrow 0.
\end{align*}
\end{corollary}

\subsubsection{Asymptotic of $\mu_k(e^{-s})$.}

Using equation \eqref{eq:QkPkFulcrumrelation derivatives}, with $m=1$, we obtain that
\begin{align}\label{eq: mean_Q_interms_P}
\mu_k(e^{-s}) = m_{k}(e^{-s})-2m_{k}(e^{-2s})\,, \quad \text{ for any }  s>0\,.
\end{align}

Define
\begin{align*}
\tilde{\mu_k}(e^{-s}) = (1-2^{-1/k})\omega_{k,1}\frac{1}{s^{1+1/k}}\,.
\end{align*}
A combination of  \eqref{eq: mean_Q_interms_P} with Lemma \ref{lem:mk} gives the following.
\begin{lemma}\label{lem:asympt to muk}For each integer $k\ge 1$,
\begin{align*}
	\mu_k(e^{-s}) = (1-2^{-1/k})\,\omega_{k,1}\,\frac{1}{s^{1+1/k}}+O(1/s)\,, \quad \text{ as } s \downarrow 0\,,
\end{align*}
and also
\begin{align*}
\frac{	\mu_k(e^{-s}) -\tilde{\mu}_k(e^{-s})}{\eta_k(e^{-s})} = O(s^{1/(2k)})\,, \quad \text{ as } s \downarrow 0\,.
\end{align*}

\end{lemma}
\begin{proof} Using \eqref{eq: mean_Q_interms_P} and Lemma \ref{lem:mk}, that is, the asymptotic equality
\begin{equation}\label{eq:mkbis}
	m_k(e^{-s}) = \omega_{k,1}\,\frac{1}{s^{1+1/k}} + O\!\left({1}/{s}\right), \qquad \text{as } s\downarrow 0\,.
\end{equation}
we find that
\begin{align*}
\mu_k(e^{-s}) = \omega_{k,1}\,\frac{1}{s^{1+1/k}}+O(1/s)-2\,\omega_{k,1}\,\frac{1}{(2s)^{1+1/k}}+O(1/2s)\,, \quad \text{ as } s \downarrow 0\,,
\end{align*}
and therefore
\begin{align*}
	\mu_k(e^{-s}) = (1-2^{-1/k})\, \omega_{k,1}\,\frac{1}{s^{1+1/k}}+O(1/s)\,, \quad \text{ as } s \downarrow 0\,.
\end{align*}
The second part of the lemma follows from equation \eqref{eq:Gkasympt}.
\end{proof}

\medskip

If we solve the equation $\mu_{k}(e^{-s_{n}}) = n$, we find that
\begin{align}\label{eq:formula for sn de Qk}
s_n = \left(\frac{(1-2^{-1/k})\,\omega_{k,1}}{n}\right)^{\frac{k}{k+1}}\,, \quad \text{ for any } n \geq 1\,.
\end{align}
We will use this sequence $s_n$, and the corresponding $\tau_n = e^{-s_n}$, in what follows.

\subsection{Hayman conditions for $Q_k$.}

Next we verify that conditions (H1) and (H2) hold for $Q_k$. Since we have already seen that the variance $\eta^2_k(e^{-s})$ tends to $+\infty$ as $s \downarrow 0$, this completes the proof that $Q_k$ is a Hayman function.

\subsubsection{Condition {\upshape{(H1)}} for $Q_k$.}
The major-arc condition for \(Q_k\) follows the same scheme as for
\(P_k\): it suffices to control the third derivative of the fulcrum
\(G_k\) on the whole vertical line. The relation~\eqref{eq:QkPkrelation}
expresses \(G_k'''\) in terms of \(F_k'''\), so the bound for \(F_k\) carries over with the very same cut
function \(h(e^{-s})=s^\alpha\).
\smallskip

From the relation \eqref{eq:QkPkFulcrumrelation}, we deduce that
$$
G_k'''(-s+\imath \theta)=F_k'''(-s+\imath \theta)-8 F_k'''(-2s+2\imath  \theta)\, , \quad \mbox{for $s>0$ and $\theta \in \R$}\,.$$
Therefore, from \eqref{eq:F3boundbis}, we have that
$$|G_k'''(-s+\imath \theta)|\le F_k'''(-s)+8F_k'''(-2s)\, , \quad \mbox{for $s>0$ and $\theta \in \R$}\,.$$

\medskip
From Lemma \ref{lem:asympt derivatives Fk} it follows that
$$F_k'''(-s)+8F_k'''(-2s)\sim (1+2^{-1/k})F_k'''(-s)\, , \quad \mbox{as $s \downarrow 0$}\,,$$
which combined with \eqref{eq:derivativesPkQk} gives a constant $C$ and $s_0>0$ so that
$$
|G_k'''(-s+\imath \theta)|\le C F_k'''(-s)\, , \quad \mbox{for every $s \in (0, s_0)$ and every $\theta\in \R$}\,.$$

For any cut function $h$ we have that
$$
\sup_{\theta \in \R} |G_k'''(-s+\imath \theta)|\, h(e^{-s})^3\le C F_k'''(-s) h(e^{-s})^3\, , \quad \mbox{for $s \in (0, s_0)$}\,.$$

With the same cut function $h$ that we have prescribed above for $P_k$, i.e., $h(e^{-s})=s^\alpha$, with the exponent $\alpha$ in the interval  $(1+1/(3k),\; 1+1/(2k))$, see equation \eqref{eq:alphainterval}, we obtain, appealing to \eqref{eq:F3boundtris}, that
$$\lim_{s \downarrow 0} \sup_{\theta \in \R} |G_k'''(-s+\imath \theta)|\, h(e^{-s})^3=0\,,$$
which by Lemma \ref{lem:majorarc} implies that (H1) holds.

\

\subsubsection{Condition {\upshape{(H2)}} for $Q_k$.}

Lemma \ref{lem:cotas de G_k en terminos de W} combined with Lemma~\ref{lema:estimate of W}, which collects lower bounds of $W$, implies the following corollary analogous to Corollary \ref{cor:bounds for char Pk}.

\begin{corollary}\label{cor:bounds for char Qk}
For any constant $\widetilde{C}>0$ there are positive constants $\widetilde{c}_1$ and $\widetilde{c}_2$ depending
only on~$k$ and~$\widetilde{C}$ such that for $s\in(0,\ln 2)$ we have that
\[
\big|\E\!\big(e^{\imath\theta\breve{Y}^{[k]}_t}\big)\big| \le
\begin{cases}
e^{-\widetilde{c}_1\theta^2}, & \text{if } |\theta|\le \widetilde{C}\,\dfrac{1}{s^{1/(2k)}}\,,\\[8pt]
e^{-\widetilde{c}_2\frac{1}{s^{1/k}}}, & \text{if } |\theta|\ge \widetilde{C}\,\dfrac{1}{s^{1/(2k)}}\,.
\end{cases}
\]
\end{corollary}

Corollary~\ref{cor:bounds for char Qk}  shows that $Q_k$ satisfies the same conclusion as Corollary~\ref{cor:bounds for char Pk}, just with different constants. Combining this with the comparability of the variance functions of Section \ref{seccion:comparability of variances}, we may deduce that condition (H2) for $Q_k$ follows exactly as condition (H2) for $P_k$ was obtained in the proof of Theorem \ref{thm:Hayman}.

\medskip

\emph{This completes the verification that $Q_k$ is in the Hayman class.}

\subsection{Asymptotic formula for $q_k(n)$.} Since $Q_k$ is in the Hayman class, $Q_k$ is strongly Gaussian, and because of Lemma \ref{lem:asympt to muk} we may apply the B\'aez-Duarte Theorem \ref{thm:baezduarte} to obtain an asymptotic formula for $q_k(n)$.

Substituting into the Báez-Duarte formula of Theorem \ref{thm:baezduarte}, the value \eqref{eq:formula for sn de Qk} for $s_n$, the mean asymptotics of Lemma~\ref{lem:asympt to muk}, the variance asymptotics for $\eta^2_k(e^{-s})$ and the asymptotics  for $Q_k(e^{-s})$ from Corollary \ref{cor: ln(Q_k)_asymptotic},  we obtain   the asymptotic formula for $q_k(n)$:
\begin{equation}\label{eq:qkasympt}
q_k(n) \sim \frac{1}{2\sqrt{\pi}} \cdot \frac{\Phi_k^{k/(2k+2)}}{(1+1/k)^{1/2}} \cdot \frac{1}{n^{(2k+1)/(2k+2)}}\, \exp\!\big\{(k+1)\,\Phi_k^{k/(k+1)}\, n^{1/(k+1)}\big\}\,, \quad \text{as } n\to\infty,
\end{equation}
where $\Phi_k = (1-2^{-1/k})\,\frac{1}{k}\,\zeta(1+1/k)\,\Gamma(1+1/k) = (1-2^{-1/k})\,\Omega_k$.

This asymptotic formula
can be traced back at least to the paper~\cite{RothSzekeres} of Roth and Szekeres, see also~\cite[Equation~23]{TranMurthyBhaduri} and~\cite{MurthyBrack}.

\subsection*{Acknowledgments and funding}

Jos\'e L.\ Fern\'andez thanks the Instituto Universitario de Matem\'aticas y Aplicaciones de la Universidad de La Laguna for its warm and
friendly hospitality during a visit made in September 2025 and also Fundaci\'on Akusmatika for generous support. Research of V.~J.\ Maci\'a was partially funded by grants PID2023-148028NB-I00 and
PID2021-123151NB-I00 from the Spanish Government. Part of this work was also supported by the Departamento de An\'alisis Matem\'atico de la Universidad de La Laguna,
through funding for several short research stays at Universidad Aut\'onoma de Madrid.

\end{document}